\documentclass[12pt]{amsart}
\usepackage{amsmath}
\usepackage{amsfonts}
\usepackage{amssymb}
\usepackage{a4wide}
\usepackage{comment}
\usepackage{bookmark}
\usepackage{graphicx}
\usepackage{xcolor}

\def \F {{\mathbb F}}

\def \Z {{\mathbb Z}}
\def \cV {{\mathcal V}}

\usepackage{todonotes}
\newcommand{\commA}[2][]{\todo[#1,color=magenta]{A: #2}}

\newcommand{\intv}[2]{\{#1,\ldots,#2\}}

\newcommand{\gen}[1]{\langle #1\rangle}

\newtheorem{theorem}{Theorem}
\newtheorem*{T1}{Theorem~\ref{thm:a}}

\newtheorem{corollary}{Corollary}
\newtheorem{proposition}{Proposition}
\theoremstyle{remark}

\newtheorem*{remark}{Remark}

\newcommand{\thmtext}{Suppose $p>3$ or $p=3$ and $n$ is even. Assume $\gamma$ is a generator of $\F_q^*$, that is $T=q-1$.
    Then for every $\varepsilon >0$, there exists a positive constant $C_\varepsilon$ such that the additive index of $d_\gamma$ is at least $C_\varepsilon \frac{q^{1-\varepsilon}}{p}$.}

\date{\today}

\begin{document}
\title[Additive index of Diffie-Hellman mapping and discrete logarithm]{On the additive index of the Diffie-Hellman mapping and the discrete logarithm }

\author{Pierre-Yves Bienvenu}

\author{Arne Winterhof}

\email{\{pierre.bienvenu,arne.winterhof\}@oeaw.ac.at}

\address{Johann Radon Institute for Computational and Applied Mathematics, Austrian Academy of Sciences, Linz, Austria}

\begin{abstract}
  Several complexity measures such as degree, sparsity and multiplicative index for cryptographic functions including the Diffie-Hellman mapping and the discrete logarithm in a finite field have been studied in the literature. In 2022, Reis and Wang introduced another complexity measure, the additive index, of a self-mapping of a finite field. In this paper, under certain conditions, we determine lower bounds on the additive index of the univariate Diffie-Hellman mapping and a self-mapping of $\F_q$ which can be identified with the discrete logarithm in a finite field. 
\end{abstract}
\maketitle

\section{Introduction}
For the whole paper, we fix a prime number $p$ and an integer $n\ge 1$ and put $q=p^n$.
Every notion of linear algebra should be understood with respect to the ground field $\F_p$.

\subsection{Diffie-Hellman mapping and discrete logarithm}
For a divisor $T$ of $q-1$, let $\gamma\not=0$ be an element of (multiplicative) order $T$
in the finite field $\F_q$ with $q$ elements. Then the {\em Diffie-Hellman mapping  to the base $\gamma$} is the bivariate function $D(X,Y)$ defined by 
$$D(\gamma^x,\gamma^y)=\gamma^{xy},\quad 0\le x,y\le T-1.$$
Any efficient way of calculating $D(\gamma^x,\gamma^y)$ for a large set of pairs $(x,y)$ would have an important impact to public-key cryptography, in particular, the Diffie-Hellman key-exchange could be attacked, see for example \cite[Chapter~3]{mov97}.
However, there are several results in the literature that this is not the case in view of several complexity measures for $D$, see in particular the monograph \cite{sh03}.
For example, it can be shown that $D(X,Y)$ cannot be represented by a low-degree polynomial (for large $T$),
see \cite{elsh01,wi01}.

The {\em univariate Diffie-Hellman mapping to the base $\gamma$} is the function $d=d_\gamma$ defined by
$$d(\gamma^x)=\gamma^{x^2},\quad 0\le x\le T-1.
$$
Since 
$$D(\gamma^x,\gamma^y)^2=d(\gamma^{x+y})d(\gamma^x)^{-1}d(\gamma^y)^{-1},\quad 0\le x,y\le T-1,$$
and both inversion in $\F_q^*=\F_q\setminus\{0\}$ and square-root finding can be done in (probabilistic) polynomial time, see for example \cite[Sections~5.4, 7.1 and 7.2]{bash96}, any fast evaluation method for~$d(\gamma^x)$ would lead to one for $D(\gamma^x,\gamma^y)$. Again any univariate polynomial which interpolates $d(X)$ on a large subset must have large degree as well as large sparsity, see \cite{mewi02,mewi08,sh03} and references therein. 

Moreover, the {\em discrete logarithm} or {\em index} of $\gamma^x$ to the base $\gamma$ is 
$${\rm ind}_\gamma(\gamma^x)=x,\quad 0\le x\le T-1.$$
If  $T<p$, in particular, if $q=p$ is a prime, we may identify the discrete logarithm with a mapping from the subgroup of $\F_q^*$ of order $T$ into $\{0,1,\ldots,T-1\}\subset \F_q$, however, for non-prime $q$ we get only the discrete logarithm modulo the characteristic $p$ of $\F_q$ and lose information about the exact integer value of the discrete logarithm. Hence, it is more natural to study the following mapping:

Let $(\beta_1,\ldots,\beta_n)$ be an ordered basis of $\F_q$ over $\F_p$ and
$$\xi_x=x_1\beta_1+\cdots+x_n\beta_n$$
whenever
$$x=x_1+x_2p+\cdots+x_np^{n-1},\quad 0\le x_i<p,$$
be an ordering of the elements of $\F_q$.
For a fixed element $\gamma$ of $\F_q^*$ of order $T$ the discrete logarithm ind$_\gamma$ can be identified with a mapping $P$ of the subgroup of~$\F_q^*$ of order $T$ into~$\F_q$ defined by
$$P(\gamma^x)=\xi_x,\quad x=0,1,\ldots,T-1.$$

More recently, the multiplicative index of $f(X)=d(X)$ and $f(X)=P(X)$ was studied in \cite{iswi22},
that is the index of the largest subgroup $S$ of $G$ such that $f(X)X^{-r}$ is constant on each coset of $S$ for some positive integer $r$.

In this paper we prove similar results on the additive index introduced by Reis and Wang in \cite{rewa22} for $d(X)$ and $P(X)$, respectively.
In fact, in the spirit of the references mentioned above, we provide lower bounds (depending on $m$) for the additive index of any self-mapping
of $\F_q$ that coincides with either of these mappings on all but $m$ elements of their domain 
$\gen{\gamma}$, that is interpolation polynomials of these two mappings.

\subsection{The additive index}
A polynomial over $\F_{q}$ of the form
$$M(X)=\sum_{j=0}^{n-1}a_jX^{p^j},\quad a_j \in \F_q,$$
is called {\em linearised} (or {\em $p$-polynomial}), see for example \cite[Section~2.1.6]{HFF}.
Note that $M(X)$ is an $\F_p$-linear map if we consider $\F_{q}$ a vector space, that is,
$$M(a\xi+b\zeta)=aM(\xi)+bM(\zeta),\quad a,b\in \F_p,\quad \xi,\zeta \in \F_{q}.$$
In fact a simple counting argument shows that any $\F_p$-endomorphism of $\F_q$ ($\F_p$-linear self-mapping of $\F_q$) is induced by a linearised polynomial; we provide the details in the appendix (Proposition \ref{prop:linmaplinpoly}).
A linearised polynomial 
is sparse and has at most $n$ non-zero coefficients. Moreover, in a normal basis representation, calculating $p^j$th powers is just a shift of coordinates and thus considered free, see for example \cite[Section~2.3]{lini83}. Hence linearised polynomials can be efficiently evaluated in $O(n)$ $\mathbb{F}_{q}$-operations in normal basis representation
and in $O(n^2\log p)$ $\mathbb{F}_{q}$-operations using repeated squaring, see \cite[Section~5.4]{bash96}, taking an arbitrary basis representation.
Here, as usual, $f=O(g)$ means $|f|\le cg$ for some constant $c>0$.
Sometimes we also use $f\ll g$ or $g\gg f$ instead. If the implied constant may depend on some parameter $\varepsilon$, we write $
f\ll_\varepsilon g$.

A mapping $F:\F_q\rightarrow\F_q$ is of {\em additive index $p^k$}
and {\em codimension} $k$ 
if there is a linearised polynomial $M(X)\in \F_{q}[X]$ 
and a linear subspace $U_0$ of dimension $n-k$ with disjoint cosets $U_i$, $i=0,1,\ldots,p^k-1$, such that
$$F(\xi)=M(\xi)+a_i,\quad \xi\in U_i,\quad i=0,1,\ldots,p^k-1,$$
for some $a_0,a_1,\ldots,a_{p^k-1}\in \F_{q}$, see  \cite[Definition~4.2]{rewa22}.

\begin{remark}
    The counting argument alluded to above reveals that $M$ may be assumed to have degree at most
    $p^{n-k-1}$, see Corollary \ref{cor:equivdef} in the appendix.

\end{remark}

 Note that a mapping of additive index $p^k$ and codimension $k$ is also a mapping of additive index $p^{k'}$ and codimension $k'$ for any $k'$ with $k\le k'\le n$.
The {\em least codimension} of $F$ is the smallest $k$ such that $F$ has codimension  $k$. Then $p^k$ is the {\em least additive index} of $F$.
Thus if the least codimension is $0$, the mapping $F$ is simply an $\F_p$-affine function.
In the contrary, if the least codimension is as large as possible,
that is $n$, the mapping $F$ shows no linear behaviour.
It is worth noting that when $q$ is large, the proportion of maps
$\F_q\rightarrow\F_q$ that have codimension strictly less than $n$ 
among all self-mappings of $\F_q$
becomes
vanishingly small, see Corollary \ref{cor:almostall} in the appendix.
A generic mapping $\F_q\rightarrow\F_q$ has therefore least codimension $n$, and a mapping of codimension $<n$ is  very much untypical, ``non-random''.
Besides degree, sparsity and multiplicative index, the additive index is therefore another measure for the complexity of a function, that is, its unpredictability and thus suitability in cryptography.

Partly we will use a slight abuse of terminology when we write additive index instead of least additive index. Moreover, the additive index of the functions $d$ or $P$, respectively, which are defined only on a subset of $\F_q$, refers to any function on $\F_q$ which coincides with $d$ or $P$ on this subset.

\subsection{Statement of some results and structure of the paper}
We have a variety of results dealing with various cases -- different regimes according to the cardinality $T$ of 
$\gen{\gamma}$ as a function of $q$, different methods according to certain arithmetic conditions on $p$ and $n$ -- so that it would be too complicated to state all results here.
Further, we are able to tolerate functions departing from the Diffie-Hellman or discrete logarithm mappings on few points; thus, we provide lower bounds for the additive index of not only
these two mappings, but of any mappings coinciding with either of them on all but $m$ elements of $\gen{\gamma}$, the quality of the bound decaying with $m$ increasing.
However, we state here some appealing and simple results, corresponding to the important cases $T=q-1$
and $m=0$. We start with the Diffie-Hellman mapping.
\begin{theorem}
  \label{thm:a}
  \thmtext
\end{theorem}
The result is even better and more general for the discrete logarithm.
\begin{theorem}
    \label{th:disclogintro}
    Suppose again $T=q-1$. Then the additive index of the discrete logarithm mapping is at least $\frac{q}{n+2}$.
\end{theorem}
We point out that we have much more general results, accommodating, under reasonable hypotheses,
$T$ as small as $q^\varepsilon$, for any $\varepsilon>0$. Further,  for the Diffie-Hellman mapping, we also have results for the case
$p=2$, and in certain cases we even obtain that the additive index is maximal, that is $q$. But to keep the introduction neat and tidy, we terminate here the exposition of our results. 

The paper is organised as follows.
We start with some preliminary results in Section~\ref{prel}.
In Section~\ref{DH} we prove lower bounds on the additive index of the univariate Diffie-Hellman mapping $d(X)$ and in Section~\ref{DL} of the mapping $P(X)$ which represents the discrete logarithm.

\section{Preliminaries}\label{prel}

\subsection{The Weil bound}

We recall the Weil bound for multiplicative character sums, see for example \cite[Theorem 5.41]{lini83}.
\begin{proposition}\label{weil}
Let $\chi$ be a multiplicative character of $\F_q$ of order $s>1$ and
$f(X)\in \F_q[X]$ be not of the $f(X)=ag(X)^s$ with $d$ different zeros (in its splitting field). Then we have
$$\left|\sum_{\xi\in \F_q}\chi(f(\xi))\right|\le (d-1)q^{1/2}.$$    
\end{proposition}
We will use the following consequence for certain character sums over subgroups.
\begin{corollary}\label{weilcor}
    Let $G$ be a subgroup of $\F_q^*$ and $\chi$ a non-trivial multiplicative character of $\F_q$.
    Then we have 
    $$\left|\sum_{\xi\in G}\chi(\xi-\xi^{-1})\right|\le 2q^{1/2}.$$
\end{corollary}
\begin{proof}
Let $T$ be the order of $G$. Then the absolute value to be estimated equals
$$ \frac{T}{q-1}\left|\sum_{\zeta\in \F_q^*}\chi(\zeta^{(q-1)/T}-\zeta^{-(q-1)/T})\right|
=\frac{T}{q-1}\left|\sum_{\zeta\in \F_q}\chi((\zeta^{2(q-1)/T}-1)\zeta^{(q-2)(q-1)/T})\right|\le 2q^{1/2},$$
where we used $\xi^{-1}=\xi^{q-2}$ and the Weil bound Proposition~\ref{weil}. 
\end{proof}

\subsection{Number of solutions to the equation $x^2\equiv j\bmod T$}

For an integer $j$ let $L_T(j)$ be the number of solutions $x$ of
$$x^2\equiv j\bmod T,\quad x=0,1,\ldots,T-1,$$
and $L_T$ be the maximum of $L_T(j)$ over all integers $j$. The following bounds on $L_T$  are well-known.
\begin{proposition}
\label{prop:sqroots}
    We have 
    $$L_T=O(\sqrt{T}).$$
    Further, if $T=rk$, where $r$ is prime and $k$ is coprime to $r$,
    then we have 
    $$L_T=O(\sqrt{k}).$$
\end{proposition}
The second statement is particularly interesting when $r$ is large.
\begin{proof}
    The first statement is a particular case of \cite[Theorem 2]{ko79},
which proves an analogous result for polynomial equations of arbitrary degree.
The second statement is a direct consequence, since
$L_T=L_r\cdot L_k$ by the Chinese remainder theorem and $L_r=2$ for
any odd prime $r$ (while $L_2=1$).
\end{proof}
\subsection{Number of squares modulo $T$}
Let $N(T)$ be the number of squares modulo $T$, that is the residues $a\in\{0,1,\ldots,T-1\}$ such that the equation
$x^2\equiv a\bmod T$ admits a solution.
By the Chinese Remainder Theorem, this is a multiplicative function;
thus $N(T)=\prod_p N(p^{v_p(T)})$ where $v_p(T)$ is the $p$-adic valuation of $T$ for every prime $p$, and it suffices to evaluate $N$ on prime powers.
We quote the following results from \cite{stangl}.
\begin{proposition}
\label{prop:stangl}
    \begin{enumerate}
        \item Let $p$ be an odd prime and $\ell\geq 1$ an integer.
        Then 
        $$N(p^\ell)=\left\lbrace \begin{array}{cc}
         \frac{p^{\ell+1}+p+2}{2(p+1)},    & \ell \text{ even}, \\[8pt]
           \frac{p^{\ell+1}+2p+1}{2(p+1)},   & \ell \text{ odd}. 
         \end{array}
         \right.$$
         \item For $p=2$ and $\ell\ge 1$ we have 
         $$N(2^\ell)=\left\lbrace \begin{array}{cc}
         \frac{2^{\ell-1}+4}{3},    & \ell \text{ even}, \\[6pt]
           \frac{2^{\ell-1}+5}{3},   & \ell \text{ odd}. 
         \end{array}
         \right.$$
    \end{enumerate}
\end{proposition}
\begin{corollary}
    \label{cor:nteps}For any odd prime $p$ and integer $\ell\geq 1$, we have
    $$N(p^\ell)\geq \frac{p^\ell}{3}\quad\mbox{and}\quad N(2^\ell)\geq \frac{2^\ell}{6}.$$
    For every $\varepsilon>0$ and integer $T\geq 1$, we have 
    $$N(T)\gg_\varepsilon T^{1-\varepsilon}.$$
\end{corollary}
\begin{proof}
    The lower bounds for $N(p^\ell)$ follow from Proposition \ref{prop:stangl}. As a result, by multiplicativity,
    $N(T)\geq T/(3^{\omega(T)}\cdot 2)$, where $\omega$ denotes the number of distinct prime factors of $T$.
    As is well known, we have $\omega(T)\ll \frac{\log T}{\log\log T}$
    so $3^{\omega(T)}\leq T^{O(1/\log\log T)}\ll_\varepsilon T^\varepsilon$. This concludes the proof.
\end{proof}
We will need the fact that no single class modulo $s$ for some $s|T$ concentrates almost all the squares modulo $T$. To make this precise, for any integer $a$ and $s$ where $s\mid T$, let $N(a,s,T)$ be the number
of squares $x^2\bmod T$ such that
$x\equiv a\bmod s$ and $N(s,T)=\max_a N(a,s,T)$.
Obviously $N(T,T)=1$ and $N(1,T)=N(T)$. For $1<s<T$ we use the following
result.
\begin{proposition}
\label{prop:squaresmodT}
For any integers $s,T$ where $s\mid T$ and $s>1$, we have
$$N(s,T)\leq \frac{8N(T)}{9}.$$
\end{proposition}
\begin{proof}
Write $s=\prod_pp^{k_p}$
and
$T=\prod_pp^{\ell_p}$, where $k_p\leq\ell_p$ for every $p$.
Then by the Chinese Remainder Theorem,
$N(s,T)=\prod_p N(p^{k_p},p^{\ell_p})$.
Thus it suffices to bound 
$N(p^k,p^\ell)$ for $\ell\geq k>0$, which in turn is at most
$N(p,p^\ell)$.
Since $N(p)\geq 2=2N(p,p)$ for all $p$, we can suppose $\ell>1$

For an odd prime $p$, we find
$N(a,p,p^\ell)=p^{\ell-1}$ if $\gcd(a,p)=1$.
Otherwise $N(a,p,p^\ell)\leq p^{\ell-2}$ since $x\equiv 0\bmod p$
implies $x^2\equiv 0\bmod p^2$.
Given the formula for $N(p^\ell)$ in Proposition~\ref{prop:stangl}, it is obvious that
$N(p^\ell)<pN(p^{\ell-1})$. Thus
$N(p^\ell)/p^{\ell-1}$ decreases with $\ell$. The limit as $\ell$ tends to infinity is $p^2/2(p+1)\geq 9/8$ for odd primes (minimal for $p=3$). 

For $p=2$, we have
$N(2,4)=1=N(4)/2$. Further for $\ell\geq 3$,
$N(0,2,2^\ell)=N(2^{\ell-2})$ since an even square mod $2^\ell$
is of the form $4y$ where $y$ is a square mod $2^{k-2}$.
A quick look at the formulas for $N(2^\ell)$ reveals that
$N(2^{\ell-2})\leq (2/3) N(2^\ell)$.
And $N(1,2,2^\ell)$ is the number of odd squares mod $2^k$, which is the number of squares congruent to 1 mod 8, which is $2^{\ell-3}$ since every number congruent to 1 mod 8 is a square mod $2^\ell$.
Again a quick look at the formulas shows that
$N(2^\ell)\geq 2^{\ell-1}/3\geq (4/3)2^{\ell-3}$ so we are done.
\end{proof}

\subsection{Number of $\xi\in G$ with $\xi-\xi^{-1}\in G$}

\begin{proposition}
\label{prop:intersInvers}
Let $G$ be a subgroup of cardinality $T$ of $\F_q^*$.
Put 
$$N=|\{\xi\in G: \xi-\xi^{-1}\in G\}|.$$
Then we have
    $$N\ge \frac{T^2}{q}-3q^{1/2}.$$
\end{proposition}
\begin{proof}
For $T=q-1$ we have 
$$N=\left\{\begin{array}{cc} q-3, & q \mbox{ odd},\\q-2,& q \mbox{ even}.\end{array}\right.$$
since $\xi-\xi^{-1}\not\in G$ if and only if $\xi^2=1$, and thus the result is trivial.  Hence, we may assume $T\le (q-1)/2$.

 Let $\chi$ be a multiplicative character of order $(q-1)/T$ of $\F_q^*$.
Then we have
$$\frac{T}{q-1}\sum_{j=0}^{(q-1)/T-1}\chi^j(\zeta)=\left\{\begin{array}{cc}
1, & \zeta \in G,\\
0, & \zeta \in \F_q^*\setminus G.
\end{array}\right.$$

     It follows that
    $$N=\sum_{\xi \in  G \setminus\{-1,1\}}
    \frac{T}{q-1}\sum_{j=0}^{(q-1)/T-1}\chi^j(\xi-\xi^{-1}).$$
(Note that $\xi=\pm 1$ is not counted by $N$.)
Separating the contribution of the trivial character $\chi^0$ we get
with
$$\epsilon=|G \cap \{-1,1\}|\in\{1,2\},$$
\begin{equation}
    \label{eq:twoinnersums}
\left|N-\frac{(T-\epsilon)T}{q-1}\right|\le \max_{j=1,\ldots,(q-1)/T-1}
\left|\sum_{\xi \in G} \chi^j(\xi-\xi^{-1})\right|.\end{equation}
(Here we used the convention $\chi(0)=0$.)
For $j=1,\ldots,(q-1)/T-1$ the absolute value of the sum over $\xi$ 
is at most $2q^{1/2}$ by Corollary~\ref{weilcor}.

Combining this bound with \eqref{eq:twoinnersums}, we now get
$$N\ge \frac{(T-2)T}{q-1}- 2 q^{1/2}$$
and the result follows since we assumed $T\le (q-1)/2$. 
\end{proof}
%

\subsection{Sums of elements of a subgroup}

We use the Minkowski sum notation: given subsets $A_1,\ldots,A_m$ of the group $(\F_q,+)$,
we define their sum set by
$$
\sum_{i=1}^mA_i=\left\{\sum_{i=1}^ma_i: (a_1,\ldots,a_m)\in\prod_{i=1}^mA_i\right\}.
$$
We simplify the notation if all $A_i$'s are equal:
$mA=\sum_{i=1}^mA$.
Similarly we define the product set $AB$
of two sets $A$ and $B$ of $\F_q$ by
$AB=\{ab:(a,b)\in A\times B\}$. Again we may iterate this construction
and get $A^m=\{a_1\cdots a_m:a_i\in A\, \forall i\}$; in theory, this notation may provoke a clash of notation with the cartesian product, but we will ensure no ambiguity arises.

We quote \cite[Theorem 5]{glru09}. A subset $A\subseteq\F_q$ is {\em symmetric} if $a\in A$ implies $-a\in A$ and {\em antisymmetric} if $a\in A$ implies $-a\not\in A$. Note that a subgroup $G$ of $\F_q^*$ is either symmetric (if $|G|$ is even and thus $-1\in G$) or antisymmetric (if $|G|$ is odd and thus $-1\not\in G$).
\begin{proposition}
    \label{prop:glru}
    Assume that $A \subset \F_q$ and 
    $B \subset \F_q$ are such that $B$ is symmetric or antisymmetric. If additionally $|A||B| > q$ then $8AB = \F_q$.
\end{proposition}

For sets $A$ of size not larger than $\sqrt{q}$, there may be no integers $m,N$ such that
$NA^m=\F_q$ because $A$ may be included in a proper subfield, or more generally
$A\subset \xi\F$ for some $\xi\in\F_q$ and a proper subfield $\F$ of $\F_q$. In such a case, $A$ is called {\em degenerate}. It turns out this degeneracy
is the only impediment to obtaining $NA^m=\F_q$ for some  integers $m,N$:
this follows from \cite[Theorem 6]{gl11}, which we reformulate below.
\begin{proposition}
\label{prop:gl11}
Let $\varepsilon\in (0,1/2)$.
    Suppose $A\subset\F_q$ is non-degenerate and
    $|A|>q^\varepsilon$.
    Let $m=\lceil \varepsilon^{-1}+1/2\rceil$ and $r= 160\cdot 6^{m-3}\left(1+\frac{\log m}{\log 2}\right)$.
    Then $rA^{2m-2}=\F_q$.
\end{proposition}
\begin{proof}
  By \cite[Theorem 6]{gl11}, if   
 $|A| >q^{1/(m-\delta)}$ for some $\delta\in (0,1)$ and integer $m\ge 3$, then we have $NA^{2m-2}=\F_q$ where
  $$N= 6^{m-3} \max\left\{30 \left(3 + \frac{\log \delta^{-1}}{\log 2}\right), 160\left(1 +  \frac{\log m}{\log 2}\right)\right\}.$$
  Choosing $\delta=\frac{1}{2}$ and $m=\lceil \varepsilon^{-1}+\delta\rceil$ we get 
  $q^{1/(m-\delta)}\le q^\varepsilon<|A|$, $\delta^{-1}\le m$ so $N\le r$ and the result follows. 
\end{proof}

Let $G$ be a subgroup of $\F_q^*$ of order $T$.
Let $r(G)$ be the smallest integer (if it exists) such that 
each element $\alpha\in \F_q^*$ can be written
as sum of $r(G)$ summands of $G$, that is,
$$\alpha=x_1+x_2+\cdots+x_{r(G)}\quad \mbox{for some }x_1,x_2,\ldots,x_{r(G)}\in G.$$
Note that for $T\ge 2$ and any $r>r(G)$ there is a solution $(x_1,x_2,\ldots,x_r)\in G^r$ of $x_1+x_2+\cdots+x_r=\alpha$ for any $\alpha\not= 0$. In the sum set notation, $r(G)$ is the smallest integer $r$ such
that $\F_q^*\subset rG$.

\begin{proposition}
\label{prop:sumset}
  For $s=2,3,\ldots,7$ we have
  $$r(G)\le s\quad \mbox{if }T> q^{1/2+1/2s}.$$
  Moreover, we have
  $$r(G)\le
  8  
  \mbox{ if }T> q^{1/2}.$$
  If $G$ is not a subset of a proper subfield of $\F_q$, that is, $T$ is not a divisor of $p^d-1$ for any divisor $d<n$ of $n$,
  then we have for any $\varepsilon$ with $0<\varepsilon<\frac{1}{2}$,
  $$r(G)\le 160\cdot 6^{m-3}\left(1+\frac{\log m}{\log 2}\right)\quad\mbox{if }T> q^\varepsilon,$$
  where $m=\lceil \varepsilon^{-1}+\frac{1}{2}\rceil,$
  that is, $r(G)$ is bounded by a constant depending only on 
  $\varepsilon$.
\end{proposition}
\begin{proof}
  For the first result see \cite[Theorem 1.1]{haio08}.
  The second result follows from 
  Proposition~\ref{prop:glru} since $G$ is either symmetric or antisymmetric 
  and the third result follows from Proposition~\ref{prop:gl11}.
\end{proof}
\section{The Diffie-Hellman mapping}
\label{DH}
In this section, we fix a prime $p$, an integer $n\geq 1$, and an element $\gamma\in\F_q^*$. Let $G=\gen{\gamma}$ and $T=|G|$. We fix $m\leq T$ and a self-mapping $F$ of $\F_q$
that coincides with the Diffie-Hellman mapping $d=d_\gamma$ on all but
at most $m$ elements of $G$.

Our goal is to prove that the least additive index of $F$ is large.
Since we have information on $F$ on only at most $T$ points, we do not expect
this index to be much more than $T$ in general. Indeed, if $k>\lceil \log_pT\rceil$,
it becomes conceivable (and very probable as $k$ grows) that there
exists a subspace $U\leq \F_q$ of codimension $k$ each of whose translates contains at most one point of $G$; then $F$ could be constant on each of
these translates.

To achieve this goal, we fix a subspace $U\leq \F_q$ and denote
by $k\leq n$ its codimension. Let $\cV=\F_q/U$ be the collection of all
cosets (aka translates) of $U$. Assume there exist $M\in\F_q[X]$ linearised
and $(a_V)_{V\in\cV}\in\F_q^\cV$ such that
\begin{equation}
\label{eq:defF}
F(\xi)=M(\xi)+a_{U+\xi}
\end{equation}
for all $\xi\in\F_q$.
We need to show that $k$ is large.
To do this, we will use frequently the following lemma.
\begin{proposition}
\label{prop:key}
    Let $t=\gcd(T,p-1)$ and $g=\gamma^{T/t}$.
    For $j\in\Z$ let
    $S_{j}=\{\xi\in \gen{\gamma}:F(g^\epsilon\xi)=d_\gamma(g^\epsilon\xi)\text{ for any }\epsilon\in\{0,j\}\}.$
    Then for any $j\in\Z$ and  $\xi=\gamma^x\in S_j$ we have
    \begin{equation}
\label{eq:key}
d(\xi)(g^{j(2x+jT/t)}-g^j)=a_{U+g^j\xi}-g^ja_{U+\xi}.
\end{equation}
In particular, if $t$ is even,
we have
\begin{equation}
\label{eq:key3}
d(\xi)((-1)^{T/2}+1)=a_{U-\xi}+a_{U+\xi} \text { for all }\xi\in S_{t/2}.
\end{equation}
\end{proposition}
Note that $g\in\F_p$, so $U+g^j\xi=g^j(U+\xi)$.
Therefore, an important feature of Equation~\eqref{eq:key} is that the right-hand side depends only on $j\bmod t$ and $U+\xi$.
\begin{proof}
    Fix $V\in\cV$, $j\in\Z$ and $\xi=\gamma^x\in U_V\cap S_j$. By Equation
    \eqref{eq:defF}, the fact that $g^j\in\F_p$ and linearity of $M$,
    we have on the one hand
 \begin{align*}
 F(g^j\xi)=g^jM(\xi)+a_{g^jV}=g^jd(\xi)+a_{g^jV}-g^ja_V
 \end{align*}
 and on the other hand
  \begin{align*}
 F(g^j\xi)&=d(g^j\xi)=\gamma^{(x+jT/t)^2}\\&=\gamma^{x^2}g^{2jx+j^2T/t}=d(\xi)g^{2jx+j^2T/t}.
 \end{align*}
Therefore 
$g^jd(\xi)+a_{g^jV}-g^ja_V=d(\xi)g^{2jx+j^2T/t}$. 
This supplies the first desired conclusion.
Now if $t$ (and hence $T$) is even
we take $j=t/2$, so
$g^j=-1$ in Equation~\eqref{eq:key} to derive~\eqref{eq:key3}.
\end{proof}
Various hypotheses on $T$ and $t$ combined with Proposition \ref{prop:key} yield remarkable results, which we state and prove separately for more legibility. Recall that $N(T)$ is the number of squares modulo $T$, equivalently the number of distinct values of $d=d_\gamma$.
\begin{theorem}
    \label{th:dh04}
    If $T\equiv 0\bmod 4$, we have $p^k\geq N(T)-2m$.
\end{theorem}

\begin{proof}
    By hypothesis
$p^n\equiv 1\bmod 4$ so $p\equiv 1\bmod 2$, whence $t$ is even.
Further $(-1)^{T/2}=1$ and
thus Equation~\eqref{eq:key3} from Proposition \ref{prop:key} implies that
$d(\xi)=(a_{-V}+a_V)/2$ is constant on $V\cap S_{t/2}$,
for every $V\in\cV$.
As a result, the number of distinct values of $F$ on $S_{t/2}$
is at most $|\cV|$, which is $p^k$.
Moreover $|G\setminus S_{t/2}|\leq 2m$.
It follows that the number $N(T)$ of distinct values of $d_\gamma$ is at most $p^k+2m$, and we are done.
\end{proof}

\begin{theorem}
\label{th:t2}
If $f>2$ is a common divisor of $T$ and $p-1$, we have $p^k\geq \frac{N(T)}{9f}-\frac{2m}{f}$.
\end{theorem}
\begin{proof}
    Let $j=t/f$ and $h=g^j$.
Let $s=f/2$ if $f$ is even and $s=f$ if $f$ is odd.
Then Equation~\eqref{eq:key} rewrites
\begin{equation}
\label{eq:key2}
d(\xi)(h^{2x+T/f}-h)=a_{hV}-ha_V \text { for all }\xi\in V\cap S_{t/f}.
\end{equation}
Whenever $\xi=\gamma^x\in V\cap S_j$ and $2x+T/f\not\equiv 1\bmod f$, 
Equation~\eqref{eq:key2} implies that
$$d(\xi)=(h^{2x+T/f}-h)^{-1}(a_{hV}-ha_V)$$ 
thus $d(\xi)$ takes only at most $s$ distinct values on $V\cap S_{t/f}$ for each $V\in\cV$.
In total, the number of distinct values of $d(\gamma^x)$
on the set of all $\gamma^x\in S_j$ where  $2x+T/f\not\equiv 1\bmod f$
is therefore at most $sp^k$.

Consider the number $M$ of squares $x^2\bmod T$ such that
$$    2x+T/f\equiv 1\bmod f.
$$
We just proved $2m+sp^k+M\geq N(T)$.
Since $M\leq N(s,T)\leq 8N(T)/9$ by Proposition \ref{prop:squaresmodT},
we conclude.
\end{proof}
We are now able to prove Theorem \ref{thm:a}, which we first restate for convenience
\begin{T1}
  \thmtext
\end{T1} 
\begin{proof}
    In view of the hypotheses, either $T=q-1\equiv 0\bmod 4$ or
    $\gcd (T,p-1)>2$, possibly both. 
    Therefore we may apply at least one of Theorem \ref{th:dh04} or
    Theorem \ref{th:t2} (with $m=0$), which, combined with Corollary \ref{cor:nteps},
    gives the desired conclusion.
\end{proof}

We have yet another result for $T=q-1$, which covers certain cases that Theorem \ref{thm:a} does not (but conversely not all cases that Theorem \ref{thm:a} does), and yields even stronger results when it applies,
namely the maximal result $p^k=q$ when $m=0$.
\begin{theorem}
    If $T=q-1$, $p>2$, and $n\not\equiv 0\bmod \frac{p-1}{2}$,
 then $p^k\geq \frac{q}{4m+1}$.
\end{theorem}

\begin{proof}
Since $T=q-1$, we have $t=\gcd(T,p-1)=p-1$.
Since $p$ is odd, $t=p-1$ and $T$ are both even.
We focus on the neutral coset
$U$ and denote by $a=a_U$ the corresponding constant.
In that case,
for any
$\xi\in U\cap S_{j}$, Equation~\eqref{eq:key} becomes
\begin{equation}
\label{eq:key4}
d(\xi)(g^{j(2x+jT/t)}-g^j)=(1-g^j)a \text { for all }\xi\in U\cap S_{j},
\end{equation}

Taking $j=t/2$ and $\xi\in U\cap S_{t/2}$ (if it exists) in \eqref{eq:key4} yields
$$\left(1+(-1)^{T/2}\right)d(\xi)=2a,$$
that is,
$$0\ne d(\xi)=a \quad \mbox{if }\frac{T}{2}\mbox{ is even},$$
and
$$a=0 \quad\mbox{if } \frac{T}{2} \mbox{ is odd}.$$

Assume that $U\cap S_{t/2}\neq\emptyset$.
If $T/2=(q-1)/2$ is odd (thus $a=0$), 
Now we consider $j=\pm 1$.
This time Equation~\eqref{eq:key} yields
\begin{equation}
\label{eq:keycong}
2x+jT/t\equiv 1\bmod t \text{ for any }j\in\{-1,1\} \text{ and } \gamma^x\in U\cap S_1\cap  S_{-1}.
\end{equation}
Taking the difference of the equations \eqref{eq:keycong} for $j=1$ and $j=-1$ yields
$T/t\equiv 0\bmod t/2$
as soon as $U\cap S_1\cap  S_{-1}\neq\emptyset$.
Further $\frac{T}{t}=\frac{p^n-1}{p-1}\equiv n\bmod p-1$,
so we find $n\equiv 0\bmod\frac{p-1}{2}$.
This is contrary to the hypothesis so either $U\cap S_{t/2}=\emptyset$
or $U\cap S_1\cap  S_{-1}=\emptyset$. Therefore $p^{n-k}-1-3m\le 0$ and we conclude.

In the case that
$\frac{T}{2}$ is even,
Equation~\eqref{eq:key4} simplifies to (using $0\ne d(\xi)=a$)
$$2jx+j^2\frac{T}{t}\equiv 0\bmod t,\quad j\in\left\{1,\frac{t}{2},-1\right\}$$
whenever $\xi=\gamma^x\in U\cap S_1\cap S_{t/2}\cap S_{-1}$.
In particular, assuming for a contradiction
$U\cap S_1\cap S_{t/2}\cap S_{t-1}\neq\emptyset$,
adding one to another the equations above for $j=1$ and $j=-1$, we get again
$$\frac{T}{t}\equiv 0\bmod \frac{t}{2},$$
thus
$n\equiv 0\bmod\frac{p-1}{2}$.
This contradicts
the result, thus $0=|U\cap S_1\cap S_{t/2}\cap S_{t-1}|\geq p^{n-k}-1-4m$
which concludes.
\end{proof}

The only cases where  none of the results above work are
\begin{itemize}
    \item $\gcd(T,p-1)=2$
and $T\equiv 2\bmod 4$,
\item $\gcd(T,p-1)=1$.
\end{itemize}
For that case, we have another method. Its applicability is not restrained by arithmetic conditions -- in contrary to all results so far, it therefore applies when $p=2$ -- but the bounds obtained are weaker, and $T$ needs to be large (at least a constant times $q^{3/4})$.
\begin{theorem}
\label{th:dhtlarge}
Suppose $T^2\geq 6q^{3/2}+6mq$.
Then  $p^k\gg T^{3/4}/\sqrt{q}$.
\end{theorem}

\begin{proof}
Let
$$
S=\{\xi\in G:\xi-\xi^{-1}\in G,\,F\text{ and } d \text{ coincide on }\{\xi,\xi^{-1},\xi-\xi^{-1}\}\}.
$$
Further, for every $(V,W)\in\cV^2$, denote
$S_{V,W}=S\cap V\cap W^{-1},$
where ${W}^{-1}=\{\xi^{-1}: \xi\in W\setminus\{0\}\}$.
Since $d(\xi)=d(\xi^{-1})$, we infer
\begin{align*}
    0=d(\xi)-d(\xi^{-1})=M(\xi)+a_V-M(\xi^{-1})-a_{W}=M(\xi-\xi^{-1})+a_V-a_W
\end{align*}
for all $\xi\in S_{V,W}$.
Therefore 
we have
$$d(\xi-\xi^{-1})=M(\xi-\xi^{-1})+a_{V-W}=a_{V-W}+a_W-a_V, \quad \mbox{for all } \xi\in S_{V,W}.$$
By the pigeonhole principle, since $\bigcup_{V,W}S_{V,W}=S$,
there exist $(V,W)\in\cV^2$ such that 
$|S_{V,W}|\geq p^{-2k}|S|$.
Moreover, for any $b\in\F_q$, the equation $\xi-\xi^{-1}=b$ has only at most two solutions, being equivalent to a quadratic equation
in the field $\F_q$.
Let $a=a_{V-W}+a_W-a_V$ and $N=|\{\zeta\in G:d(\zeta)=a\}|$.
This is at most the maximal number of square roots of an integer $j$ modulo $T$, which is denoted by $L_T$.
Further,
the argument above yields
$$L_T\geq N\geq \frac{1}{2}|S_{V,W}|\ge \frac{1}{2p^{2k}}\left(|\{\xi\in G:\xi-\xi^{-1}\in G\}|-3m\right).$$
Applying Proposition \ref{prop:intersInvers} and Proposition \ref{prop:sqroots},
we find
\begin{equation}
\label{eq:lt}
T^{1/2}\gg L_T\geq\frac{1}{2p^{2k}}\left(\frac{T^2}{q}-3q^{1/2}-3m\right).
\end{equation}
By hypothesis, the term in the parenthesis is at least $T^2/(2q)$,
so that $T^{1/2}\gg T^2/(qp^{2k})$ and we conclude.
\end{proof}
\begin{remark}
    In the important case where $T$ is a prime, or more generally $T=rb$
    where $r$ is a prime and $b$ is bounded, we have $L_T=O(\sqrt{b})=O(1)$ so Equation~\eqref{eq:lt} gives the much better bound
    $p^k\gg T/\sqrt{q}$ provided that $T^2\geq 6q^{3/2}+6mq$.
\end{remark}

Theorem \ref{th:dhtlarge} requires $T\gg q^{3/4}.$
We supply one last result, again free of any arithmetic constraints on $T$ and $p$, that kicks in as soon as $T\gg q^{1/2}$. 
Our argument seems unable to accommodate the $m$ values where $F$ and $d$ may disagree, so we will assume
$m=0$ from now on, thus $F=d$ on $G$.
\begin{proposition}
\label{prop:summands}
Let $r\geq 1$ be an integer.
Let $(V_i,W_i)_{i=1,\ldots,r}\in\cV^{2r}$.
    The map 
    $$(\xi_1,\ldots,\xi_r)\mapsto F\left(\sum_{i=1}^r (\xi_i-\xi_i^{-1})\right)$$
    is constant on $\prod_{i=1}^r V_i\cap W_i^{-1}$.
\end{proposition}
\begin{proof}
    Let $(\xi_1,\ldots,\xi_r)\in \prod_{i=1}^r V_i\cap W_i^{-1}$.
    Then \begin{align*}
       F\left(\sum_{i=1}^r (\xi_i-\xi_i^{-1})\right)&=M\left( \sum_{i=1}^r (\xi_i-\xi_i^{-1})\right)+a_{\sum_{i=1}^r(V_i-W_i)}\\
       &=\sum_{i=1}^r(M(\xi_i)-M(\xi_i^{-1}))+a_{\sum_{i=1}^r(V_i-W_i)}\\
       &=\sum_{i=1}^r(F(\xi_i)-a_{V_i}-F(\xi_i^{-1})+a_{W_i})+a_{\sum_{i=1}^r(V_i-W_i)}\\
       &=\sum_{i=1}^r(d(\xi_i)-d(\xi_i^{-1}))+a_{\sum_{i=1}^r(V_i-W_i)}+\sum_{i=1}^r (a_{W_i}-a_{V_i})\\
       &=a_{\sum_{i=1}^r(V_i-W_i)}+\sum_{i=1}^r (a_{W_i}-a_{V_i}).
    \end{align*}
    The last member of this sequence of equalities clearly is constant
    on $\prod_{i=1}^r V_i\cap W_i^{-1}$ so we are done.
\end{proof}
Recall the notation $hB=\{b_1+\ldots+b_h:b_1,\ldots,b_h\in B\}$.
\begin{corollary}
    \label{cor:hB}
    Let $B=\{\xi-\xi^{-1}:\xi\in G\}$.
    Suppose $G\subset hB$ for some $h\ge 1$.
    Then $$p^k\geq N(T)^{1/(2h)}.$$
\end{corollary}
\begin{proof}
    Applying Proposition \ref{prop:summands} with $r=h$,
    we conclude that $d$ takes on at most $p^{2hk}$ pairwise distinct values on
    $G$. Since by definition $d$ takes on exactly $N(T)$ pairwise distinct values,
    we deduce that $p^{2hk}\ge N(T)$ as desired.
\end{proof}
\begin{theorem}
    Suppose  $T> 2 q^{1/2}$. Then $p^k\geq N(T)^{1/32}\gg_\varepsilon T^{1/32-\varepsilon}$ for all $\varepsilon>0$.
\end{theorem}
\begin{proof}

    Let $B=\{\xi-\xi^{-1}:\xi\in G\}$ and $A=\{\xi+\xi^{-1}:\xi\in G\}$.
    Then $|A|\geq |G|/2$ and $|B|\geq |G|/2$,
    since for every $a\in\F_q$, both equations $\xi-\xi^{-1}=a$ and $\xi+\xi^{-1}=a$ admit at most two solutions.
    Further $B$ is symmetric (that is $B=-B$).
    By hypothesis 
    $$|A|\cdot |B|\geq \frac{T^2}{4}>q,$$ 
    which implies
    $8AB=\F_q$ by Proposition \ref{prop:glru}.

    Note that $(\xi-\xi^{-1})(\zeta+\zeta^{-1})
    =\xi\zeta-(\xi\zeta)^{-1}+\xi\zeta^{-1}-(\xi\zeta^{-1})^{-1}$,
    so that $AB\subset B+B$.
    In particular, we infer that $G\subset 16B$.
    We conclude by an appeal to Corollary \ref{cor:hB} and Corollary~\ref{cor:nteps}.
\end{proof}
To conclude this section, we state one more result which covers the cases
where $T$ is as small as $q^\varepsilon$.
\begin{theorem}
Suppose $A=\{\xi+\xi^{-1}:\xi\in G\}$ is not included in a proper subfield of $\F_q$.
    For every $\varepsilon>0$,
    there exists a constant $c=c(\varepsilon)>0$ such that
    $T\geq 2q^\varepsilon$ implies $p^k\ge T^c$.
\end{theorem}
\begin{proof}
    As above,
    let $$B=\{\xi-\xi^{-1}:\xi\in G\}\quad\mbox{and}\quad A=\{\xi+\xi^{-1}:\xi\in G\}.$$
    Then, as already seen, $|A|\geq |G|/2\gg q^\varepsilon$ and $|B|\geq |G|/2$.
    By Proposition \ref{prop:gl11},
    we have $rA^{2m-2}=\F_q$, with $r$ and $m$ depending on $\varepsilon$
    defined in that proposition.
    A fortiori, $rA^{2m-2}B=\F_q$.
    Note that 
    $$(\xi+\xi^{-1})(\zeta+\zeta^{-1})
    =\xi\zeta+(\xi\zeta)^{-1}+\xi\zeta^{-1}+(\xi\zeta^{-1})^{-1},$$
    so that $A^2\subset A+A$ and by induction $A^k\subset 2^{k-1}A$ for all $k\ge 1$.
    Further, $$A^kB\subset B\cdot 2^{k-1}A\subset 2^{k-1}BA$$
    and since we saw above that $AB\subset 2B$
    we conclude $A^kB\subset 2^kB$ for every $k\ge 1$, so finally
    $(r2^{2m-2})B=\F_q$.
     Again we conclude by an appeal to Corollary \ref{cor:hB} and Corollary~\ref{cor:nteps}.
\end{proof}
\begin{remark}
\begin{enumerate}
    \item 
    The proof provides a constant $c=c(\varepsilon)$ which decays
    exponentially with~$\varepsilon$, that is $c=\exp(-O(\varepsilon^{-1}))$.
    \item The condition
    $A=\{\xi+\xi^{-1}:\xi\in G\}$ is not included in a proper subfield of $\F_q$ is equivalent to the condition: for every proper divisor $d$ of $n$,
    $T$ divides neither $p^d-1$ nor $p^d +1$.
    Indeed, a proper subfield of $\F_q$ is a field of the form $\F_{p^d}$ for some $d\mid n$, $d<n$.
We have the sequence of equivalent assertions
\begin{align*}
    \xi+\xi^{-1}\in \F_{p^d} &\Leftrightarrow (\xi+\xi^{-1})^{p^d}=\xi+\xi^{-1}\\
&\Leftrightarrow\xi^{p^d}+\xi^{-p^d}=\xi+\xi^{-1}\\
&\Leftrightarrow \xi^{2p^d}+1-\xi^{p^d+1}-\xi^{p^d-1}=0\\
&\Leftrightarrow (\xi^{p^d-1}-1) (\xi^{p^d+1}-1)=0,
\end{align*}
so $A\subset \F_{p^d}\Leftrightarrow T\mid p^d-1\text{ or }T\mid p^d +1$ as claimed.
Observe that $T\mid p^d-1\Leftrightarrow G\subset \F_{p^d}$.
Observe additionally that since $T\mid q-1=p^n-1$,
the property $T\mid p^d +1$ for some $d\mid n$ is only possible if $n/d$ and thus $n$ is even or $T\le 2$.
So if $n$ is odd and $T>2$, $A\subset \F_{p^d}\Leftrightarrow G\subset \F_{p^d}$.
\end{enumerate}
\end{remark}
\section{Discrete logarithm}
\label{DL}

Now we consider the mapping $P$ described in the introduction, which can be identified with the discrete logarithm.
Fix a prime~$p$, an integer $n\geq 1$
and $T\mid q-1$ where $q=p^n$. 
Let $\gamma\in\F_q^*$ have multiplicative order $T$ and let $P:\gen{\gamma}\rightarrow\F_q$ be the corresponding discrete logarithm mapping.
\subsection{Some results on the mapping $x\mapsto \xi_x$}
First we need a couple of lemmas about the associated mapping 
$x\mapsto \xi_x$. Recall that this mapping, described in the introduction, depends on the choice, which we henceforth assume fixed, of
an ordered basis $(\beta_1,\ldots,\beta_n)$ of $\F_q$ as a vector space over $\F_p$.
\begin{proposition}
\label{xix1}
     For $x=0,1,\ldots,q-2$
    the difference
$\zeta_x=\xi_{x+1}-\xi_x$ 
is
$$
\zeta_x=\sum_{i=1}^j\beta_i,$$
where $j$ is the minimum in $\{1,2,\ldots,n\}$ with $x_j\le p-2$ in the $p$-adic expansion
$$x=x_1+x_2p+\cdots+x_np^{n-1},\quad 0\le x_i<p,$$
of $x$.
\end{proposition}
\begin{proof}
 Note that the minimum $j$ exists since $x\le q-2$.
 We have
\begin{eqnarray*}
    x&=&(p-1)(1+p+\cdots+p^{j-2})+x_jp^{j-1}+x_{j+1}p^j+\cdots +x_np^{n-1}\\
    &=&p^{j-1}-1+x_jp^{j-1}+x_{j+1}p^j+\cdots +x_np^{n-1}
\end{eqnarray*}
and 
$$x+1=(x_j+1)p^{j-1}+x_{j+1}p^j+\cdots +x_np^{n-1}.$$
We find
$$\xi_x=-\sum_{i=1}^{j-1}\beta_j+\sum_{i=j}^nx_i\beta_i,$$
$$\xi_{x+1}=(x_j+1)\beta_{j}+\sum_{i=j+1}^nx_i\beta_i,$$
and thus 
$$\xi_{x+1}=\xi_x+\sum_{i=1}^j\beta_i,$$
which concludes the proof.
\end{proof}
Thus, the difference $\xi_{x+1}-\xi_x$ has at most $n$ possible values.
In general, the difference $\xi_{x+a}-\xi_x$, for a fixed $a$,
can take on more distinct values, but at most $2^{n-1}$ as the next proposition shows.
\begin{proposition}
\label{xixa}
    Fix a positive integer $a\in\intv{0}{T-1}$.
    \begin{enumerate}
        \item Suppose $0\leq x<T-a$.
        Write the base $p$ decomposition
    $$a=a_1+a_2p+\cdots+a_np^{n-1},\quad 0\le a_i<p.$$
    Then 
    $$P(\gamma^{x+a})-P(\gamma^x)=\xi_{x+a}-\xi_x\in a_1\beta_1+\sum_{j=2}^n\{a_j\beta_j,(a_j+1)\beta_j\}.$$
    If additionally $a_1=a_2=\cdots=a_j=0$ for some $j\in\intv{0}{n}$, the coordinate of $\xi_{x+a}-\xi_x$ along
    $\beta_{j+1}$ in the basis $(\beta_1,\ldots,\beta_n)$ is $a_{j+1}.$
    \item 
    If $T-a\leq x<T$, let
    $$b=T-a=b_1+b_2p+\cdots+b_np^{n-1},\quad 0\le b_i<p.$$
    Then
    $$P(\gamma^{x+a})-P(\gamma^x)=\xi_{x+a-T}-\xi_x\in -b_1\beta_1+\sum_{j=2}^n\{-b_j\beta_j,-(b_j+1)\beta_j\}.$$
    If additionally $b_1=b_2=\cdots=b_j=0$ for some $j\in\intv{1}{n}$, the coordinate of $\xi_{x+a-T}-\xi_x$ along
    $\beta_{j+1}$ in the basis $(\beta_1,\ldots,\beta_n)$ is $-b_{j+1}.$\end{enumerate}
\end{proposition}
\begin{proof}
We prove the first claim.
Write the base $p$ decomposition
    $$x=x_1+x_2p+\cdots+x_np^{n-1},\quad 0\le x_i<p.$$
    Let $j\in\intv{1}{n}$.
    If 
    $$\sum_{i=1}^jx_ip^{i-1}+\sum_{i=1}^ja_ip^{i-1}<p^j,$$
    then the digit of $p^j$ in $x+a$ is $x_{j+1}+a_{j+1}\in \Z$ or $x_{j+1}+a_{j+1}-p\in \Z$; otherwise it is
    $x_{j+1}+a_{j+1}+1\in \Z$ or $x_{j+1}+a_{j+1}+1-p\in \Z$.
    Therefore the $j$th coordinate of $\xi_{x+a}$ in the base $(\beta_1,\ldots,\beta_n)$ is either $x_j+a_j\in \F_p$ or $x_j+a_j+1\in \F_p$. Since $\sum_{i=1}^jx_ip^{i-1}<p^j$, the first case is always true when
    $a_1=a_2=\cdots=a_j=0$.
    This concludes the proof of the first assertion.

    For the second one, write $y=x+a-T=x-b$, so $0\leq y<a=T-b$.
    Then $\xi_{x+a-T}-\xi_x=-(\xi_{y+b}-\xi_y)$ and we can apply the first claim.
\end{proof}
\subsection{Bounds for the additive index of the discrete logarithm}
For this section, we
fix an element $\gamma\in \F_q$ of order $T$, an integer $m$ with
$0\le m\leq T$ 
and a self-mapping $F$ of $\F_q$
that coincides with $P$ on all but
at most $m$ elements of $G=\gen{\gamma}$.
Our goal is to show that the least additive index of $F$ is large.

Like for the Diffie-Hellman mapping, since we have information on $F$ on only $T$ points, we do not expect
this index to be much more than $T$ in general. 

To achieve this goal, we fix a subspace $U\leq \F_q$ and denote
by $k\leq n$ its codimension. Let $\cV=\F_q/U$ be the collection of all
cosets (aka translates) of $U$. Assume there exist $M\in\F_q[X]$ linearised
and $(a_V)_{V\in\cV}\in\F_q^\cV$ such that
\begin{equation}
\label{eq:defF2}
F(\xi)=M(\xi)+a_{U+\xi}
\end{equation}
for all $\xi\in\F_q$.
We need to show that $k$ is large.

First, we get an excellent result in the case where $T=q-1$, that is, $\gamma$ is a generator of~$\F_q^*$. In the case $m=0$, we recover Theorem \ref{th:disclogintro}.
\begin{theorem}
Suppose $T=q-1$.
    Then  
    $$p^k\ge \frac{q}{n+2m+2}.$$
\end{theorem}
\begin{proof}
By Equation~\eqref{eq:defF2},
for any $V\in\cV$ and
 $\gamma^x\in U,\gamma^z\in V$,
we have
\begin{align*}
    F(\gamma^x+\gamma^z)&=M(\gamma^x+\gamma^z)+a_V\\
    &=M(\gamma^x)+M(\gamma^z)+a_V\\
    &=F(\gamma^x)-a_U+F(\gamma^z).
\end{align*}
Let $y\in\Z$ be such that $\gamma-1=\gamma^y$.
When $z=y+x$, the calculation above yields
\begin{equation}
    \label{eq:gamxz}
    F(\gamma^{x+1})=F(\gamma^x+\gamma^{y+x})=F(\gamma^x)-a_U+F(\gamma^{y+x}).
\end{equation}
Let $$S=\{x\in\intv{0}{T-2}:\gamma^x\in U,\, F\text{ and } P\text{ coincide on both }\gamma^x,\gamma^{x+y}\}.$$
By hypothesis, $|S|\geq |U|-2-2m=p^{n-k}-2-2m$.


When $x\in S$, Equation~\eqref{eq:gamxz} gives
\begin{equation}
    \label{eq:leftright}
\xi_{x+1}-\xi_x+a_U=P(\gamma^{y+x}).
\end{equation}
When $x$ ranges in $S$, the element $\xi_{x+1}-\xi_x$, and hence the left-hand side of Equation~\eqref{eq:leftright} can take only
at most $n$ values by Proposition \ref{xix1}.
The right-hand side, by injectivity of the discrete logarithm, takes on exactly $|S|$ values.
Thus $p^{n-k}-2-2m\le |S|\le n$ which concludes.
\end{proof}

We can work with smaller $T$, representing $\gamma-1$ as a sum of elements from $\gen{\gamma}$, but the quality of the bound decays quickly with the number of summands required. The proof runs initially along similar lines as above.
\begin{theorem}
\label{th:discloggen}
    Assume $$\gamma-1\in rG=\left\{\sum_{i=1}^r \xi_i:(\xi_1,\ldots,\xi_r)\in G^r\right\},$$ where $r$ is an integer coprime to $p$.
    Suppose $m<\frac{T-1}{r+2}$. 
    Then
    $$p^k\ge \left(\frac{T-(r+2)m-1}{rn}\right)^{1/(r+1)}.$$
\end{theorem}
\begin{proof}
Assume the contrary thus 
\begin{equation}\label{Tinq}
\frac{T-(r+2)m-1}{p^{(r+1)k}}>rn.
\end{equation}
By Equation~\eqref{eq:defF2}, for any integers $x,z_1,\ldots,z_r$, we have
\begin{equation}
    \label{eq:xyi}
F\left(\gamma^x+\sum_{i=1}^r\gamma^{y_i}\right)=F(\gamma^x)+\sum_{i=1}^rF(\gamma^{z_i})+a_{\sum_{i=1}^rV_i+W}-\sum_{i=1}^ra_{V_i}-a_W,
\end{equation}
where
$W,V_1,\ldots,V_r$ are the cosets of $U$ containing $\gamma^x,\gamma^{z_1},\ldots,\gamma^{z_r}$, respectively.
    By hypothesis $\gamma-1\in rG$.
Therefore, let $0\leq y_1\leq \cdots\leq y_r\leq T-1$ be such that
$\gamma^{y_1}+\cdots+\gamma^{y_r}=\gamma-1$.
Then $\gamma^{y_1+x}+\cdots+\gamma^{y_r+x}=\gamma^{x+1}-\gamma^x$ for any integer $x$,
so that Equation~\eqref{eq:xyi}
gives
\begin{equation}
    \label{eq:xx2}
    F(\gamma^{x+1})=F(\gamma^x)+\sum_{i=1}^rF(\gamma^{y_i+x})+a_{\sum_{i=1}^rV_i+W}-\sum_{i=1}^ra_{V_i}-a_W,
\end{equation}
where again
$W,V_1,\ldots,V_r$ are the classes modulo $U$ of $\gamma^x,\gamma^{y_1+x},\ldots,\gamma^{y_r+x}$, respectively.
Let $$S=\{x\in\intv{0}{T-2}:F\text{ and } P\text{ coincide on all of }\gamma^x,\gamma^{x+1},\gamma^{y_1+x},\ldots,\gamma^{y_r+x}\}.$$
By definition, $|S|\geq T-1-(r+2)m$.
When $x\in S$, Equation~\eqref{eq:xx2} becomes
\begin{equation}
    \label{eq:xx1}
    \xi_{x+1}=\xi_x+\sum_{i=1}^rP(\gamma^{y_i+x})+a_{\sum_{i=1}^rV_i+W}-\sum_{i=1}^ra_{V_i}-a_W,
\end{equation}
When $x$ ranges in $S$,
the $(r+2)$-tuple
$(U+\gamma^x,U+\gamma^{y_1+x},\ldots,U+\gamma^{y_r+x},\xi_{x+1}-\xi_x)$
can take at most $p^{(r+1)k}n$ values in $(\F_q/U)^{r+1}\times\F_q$ by Proposition \ref{xix1}.
By the pigeonhole principle, since
$|S|/(p^{(r+1)k}n)>r$ by \eqref{Tinq},
one of these values $(W,V_1,\ldots,V_r,v)$, which we henceforth fix, is taken at least $(r+1)$ times.
Thus by Equation~\eqref{eq:xx1}, there exist $0\leq x_0<\cdots<x_{r}<T$ such that
\begin{equation}
    \label{eq:sumxiequal}
\sum_{i=1}^rP(\gamma^{y_i+x_0})=\sum_{i=1}^rP(\gamma^{y_i+x_1})=\cdots=\sum_{i=1}^rP(\gamma^{y_i+x_{r}})
=v-a_{\sum_{i=1}^rV_i+W}+\sum_{i=1}^ra_{V_i}+a_W.\end{equation}
When $r=1$, this simply yields
$P(\gamma^{y_1+x_0})=P(\gamma^{y_1+x_1})$
which, by injectivity of the map 
$x\mapsto P(\gamma^x)$ on the interval $\intv{0}{T-1}$
and periodicity, implies $x_1\equiv x_0\bmod T$, a contradiction.

The general case is more intricate.
Denote  $y_i^{(j)}=y_i+x_j\in\intv{0}{2T-2}$, for $i=1,\ldots,r$, 
and $j=0,\ldots,r$.
For any fixed $j$, this is a non-decreasing sequence in $i$; and for any fixed~$i$, it is an increasing sequence in $j$.
Let $f(i,j)=\left\lfloor\frac{y_i^{(j)}}{T} \right\rfloor$ be the Euclidean quotient of $y_i^{(j)}$ by~$T$.
Thus the sequence
$i\mapsto f(i,j)$ is non-decreasing. Also the sequence $j\mapsto f(i,j)$
is non-decreasing. Further $0\leq f(i,j)\leq 1$ for every $j$.
Therefore, for every $j$, there exist $\ell_j\in\intv{0}{r}$
such that
\begin{equation*}
    f(i,j)=\left\lbrace 
    \begin{array}{lc}
0, & i\leq \ell_j \\ 
1, & i>\ell_j
\end{array} 
    \right.
\end{equation*}
and the sequence $(\ell_j)_{j\in\intv{0}{r}}$ is non-increasing.
We distinguish two cases.
\begin{enumerate}
    \item There exist $j$
    such that $\ell_j=\ell_{j+1}$.
    Thus $f(i,j)=f(i,j+1)$ for all $i$.
Write $z_i=y_i^{(j)}$ and $b=x_{j+1}-x_j$, so that $z_i=y_i+x_j$ and $z_i+b=y_i+x_{j+1}$ have the same quotient~$f(i,j)$ in the Euclidean division by $T$.
Then taking the difference of
the $(j+2)$th and $(j+1)$th term from the left in Equation~\eqref{eq:sumxiequal}, we have
$$
\sum_{i=1}^r(P(\gamma^{z_i+b})-P(\gamma^{z_i}))=0
$$
and $z_i$ and $z_i+b$ have the same quotient $q_i$ in the Euclidean division by $T$ for each~$i$.
Thus 
$$
\sum_{i=1}^r(\xi_{z_i-q_iT+b}-\xi_{z_i-q_iT})=0,
$$
where $0\leq z_i-q_iT<T-b$ for each $i$.
Applying Proposition \ref{xixa}, with the base~$p$ decomposition
$b=\sum_{i=1}^nb_ip^{i-1}$,
we find that 
$rb_1\equiv 0\bmod p$, whence $b_1=0$.
Applying Proposition \ref{xixa} repeatedly, we find successively that all the digits of $b$ are 0 so that $b=0$, a contradiction.
\item The sequence $\ell_j$ is decreasing, hence
 $\ell_0=r$ and $\ell_r=0$.
 We infer $f(i,0)=0$ and $f(i,r)=1$ for all $i\in\intv{1}{r}$, in particular, $y_i+x_r\ge T$ and $y_i+x_0<T$.
 
 Write $z_i=y_i^{(0)}$ and $b=x_{r}-x_0$, so that $T-b\leq z_i<T$ for all $i$.
 Then
 taking the difference of
the $(r+1)$th term and the first term from the left in Equation~\eqref{eq:sumxiequal},
 we have
$$
\sum_{i=1}^r(P(\gamma^{z_i+b})-P(\gamma^{z_i}))=\sum_{i=1}^r(\xi_{z_i+b-T}-\xi_{z_i})=0
$$
and we can apply Proposition \ref{xixa} repeatedly as above to obtain
$b=T$, which is also a contradiction.
\end{enumerate}
This finishes the proof.
\end{proof}

\begin{remark}
    In view of the proof and Proposition \ref{xix1},
    we can replace the $n$ in the bound on the additive index of $F$
    by $1+\lfloor\log_p(T-1)\rfloor$; since this improves only moderately the bound but makes it more difficult to read,
    we preferred stating the theorem the way we did.
\end{remark}
\begin{corollary}
    Assume $\gamma$ has multiplicative order $T>q^{3/4}$.
    Then the additive index of the discrete logarithm mapping $P$ to the base $\gamma$
    is at least $$\left(\frac{T-1}{2n}\right)^{1/3}\quad\mbox{if}\quad  p>2$$ and at least 
    $$\left(\frac{T-1}{3n}\right)^{1/4}\quad \mbox{if}\quad p=2.$$
\end{corollary}
\begin{proof}
    Because $T>q^{3/4}$,  Proposition \ref{prop:sumset} provides
   $\F_q^*\subset  G+G$. Thus $\gamma-1\in 2G$ and also $\gamma-1\in 3G$. 
    We therefore apply Theorem \ref{th:discloggen}, with $r=2$ (and $m=0$) if $p\neq 2$ and $r=3$ if $p=2$, 
    and conclude.
\end{proof}

\begin{corollary}
    Assume $T>q^{1/2}$.
    Then the additive index
    of the discrete logarithm mapping
    is at least 
    $$\left(\frac{T-1}{8n}\right)^{1/9}\quad \mbox{if}\quad p>2,$$ and
    at least
    $$\left(\frac{T-1}{9n}\right)^{1/10}\quad \mbox{if}\quad p=2.$$
\end{corollary}
\begin{proof}
    Indeed, in this case $\F_q=8 G$ by Proposition \ref{prop:sumset}, so we can apply Theorem \ref{th:discloggen} with $r=8$ or $r=9$, respectively.
\end{proof}
\begin{corollary}
For every $\varepsilon>0$, there is a constant integer $C_\varepsilon>0$ such that the following holds.
    Assume $T>q^\varepsilon$
    and $T$ does not divide any of the numbers $p^d-1$ for $d\mid n,d<n$.
    Then the additive index of the discrete logarithm mapping is at least 
    $$\left(\frac{T-1}{C_\varepsilon n}\right)^{1/(C_\varepsilon+1)}.$$
\end{corollary}
\begin{proof}
    Again this is an immediate consequence of Proposition \ref{prop:sumset} and Theorem \ref{th:discloggen}.
\end{proof}
Our next theorem, Theorem \ref{th:oldlog}, provides a different lower bound for the additive index of the discrete logarithm.
When it applies, Theorem \ref{th:oldlog} is usually superior to 
Theorem \ref{th:discloggen}. But it applies only under the arithmetic condition $\gcd(T,p-1)>1$,
which in particular requires $p>2$. The method of proof is a combination of the one used so far in this section and the method used for Theorems \ref{th:dh04} and
\ref{th:t2}.
\begin{theorem}
\label{th:oldlog}
    Assume $\gcd(T,p-1)>1$. 
    Then  $p^k\ge \frac{T-2m}{2^n}$.
\end{theorem}
\begin{proof}
    Let $t=\gcd(T,p-1)>1$ and
    $g=\gamma^\frac{T}{t}$, thus $g\in G\cap \F_p^*$ and $g\neq 1$.
    Let $$S=\{\xi\in G:F\text { and }P\text{ coincide on } \{\xi,g\xi\}\}.$$ 
    By hypothesis,
    $|S|\ge T-2m$.
    By Proposition \ref{xixa}, the expression
    $$P(g\gamma^x)-P(\gamma^x)=P(\gamma^{x+\frac{T}{t}})-P(\gamma^x)$$ 
    takes on at most $2^n$ pairwise distinct values as $x$ ranges in $\Z$.
    In particular $P(g\xi)-P(\xi)$ takes on at most $2^n$ pairwise distinct values as $\xi$ ranges in
    $G$.
    In addition,
    by Equation~\eqref{eq:defF2},
    we have
    $P(g\xi)=gP(\xi)-a_{gV}+ga_{V}$ for every $V\in\cV,\,\xi\in V\cap S$.
    Thus, for any $\xi\in S$,
    we have
    \begin{equation}
        \label{eq:pxi}
        P(\xi)=(g-1)^{-1}(P(g\xi)-P(\xi)+a_{gV}-ga_V)
    \end{equation}
    where $V=U+\xi$.
    As $\xi$ ranges in $G$, the left-hand side of \eqref{eq:pxi} admits exactly $|S|\ge T-2m$ pairwise distinct values since the discrete logarithm is 
injective, whereas the right-hand side takes on at most $2^np^k$ distinct values.
We infer $p^k\ge (T-2m)/2^n$.
\end{proof}

\section*{Appendix: Linearised polynomials and linear maps}
\begin{proposition}
\label{prop:linmaplinpoly}
    Let $\F_q$ be a finite field of characteristic $p$, where $q=p^n$ and $n\ge 1$.
    Let $r\leq n$ be an integer.
    Let $U\subset \F_q$ be an $\F_p$-linear subspace of $\F_q$ of dimension $r$.
    Then the set $V$ of all $\F_p$-linear maps $U\rightarrow\F_q$ is equal to the set $W$ of all maps
    $U\rightarrow\F_q$ that are induced by linearized polynomials of degree at most $p^{r-1}$.
\end{proposition}

\begin{proof}
First we observe that both sets are subspaces of the $\F_q$-vector space $\F_q^{U}$ of all maps $U\rightarrow\F_q$.
    Next, we point out that a map $U\rightarrow\F_q$ induced by a linearised polynomial is linear,
    which provides one inclusion $W\subset V$ of a subspace in another.
    It suffices to prove the equality of dimensions to conclude.
    The dimension of $V$ is $r$.
Thus it suffices to show that the linear map
$$
\F_q^r\rightarrow W,\quad a=(a_0,\ldots,a_{r-1})\mapsto F_a=(x\mapsto\sum_{i=0}^{r-1} a_ix^{p^i})
$$
is injective.
Now let $a=(a_0,\ldots,a_{r-1})$ be such that $F_a=0$.
Thus $\sum_{i=0}^{r-1} a_ix^{p^i}=0$ for all $x$ in $U$.
Then the polynomial $\sum_{i=0}^{r-1} a_iX^{p^i}$ of degree at most $p^{r-1}$ admits $|U|>p^{r-1}$ roots, which implies that this is the zero polynomial, whence $a=0$ and we are done.
\end{proof}
\begin{corollary}
\label{cor:equivdef}
    Let $f$ be a map $\F_q\rightarrow\F_q$.
    The following are equivalent.
    \begin{enumerate}
        \item $f$ has codimension $k$.
        \item There exists a subspace $U\leq \F_q$ of codimension $k$, a linear map $\ell: U\rightarrow \F_q$ and constants $b=(b_i)_{i=0,\ldots,p^{k-1}}$,
    such that,
    given a system $(v_i)_{i=0,\ldots,p^{k-1}}$ of representatives of $\F_q$ modulo $U$, we have
    $$\forall i\in\intv{0}{p^k-1},\,
    f(u+v_i)=\ell(u)+b_i.
    $$
    \item There exists a subspace $U\leq \F_q$ of codimension $k$, a linearized polynomial $M\in\F_q[X]$ of degree at most $p^{n-k-1}$
    and constants $(a_i)_{i=0,\ldots,p^{k-1}}$
    such that, denoting by $U_i,i=0,\ldots,p^{k-1}$ the cosets of $U$,
    we have $f(\xi)=M(\xi)+a_i$ for any $i$ and $\xi\in U_i$.
    \end{enumerate}
    
\end{corollary}
\begin{proof}
$(1)\Rightarrow (2):$
    If $f$ has codimension $k$,
    there exists a subspace $U\leq \F_q$ of codimension $k$,
    whose cosets are denoted by $U_i,i=0,\ldots,p^{k-1}$
    a linearized polynomial $M$ and constants $(a_i)_{i=0,\ldots,p^{k-1}}$
    such that for all $\xi\in U_i$ we have $f(\xi)=M(\xi)+a_i$.
    Given a system $(v_i)_{i=0,\ldots,p^{k-1}}$ of representatives of $\F_q$ modulo $U$, which may be ordered in such a way that
    $U_i=U+v_i$ for each $i$, we therefore have
    $f(u+v_i)=M(u+v_i)+a_i=M(u)+M(v_i)+a_i=M(u)+b_i$, for each $i$ as desired,
    where $b_i=a_i+M(v_i)$. Since $u\mapsto M(u)$ is a linear map $U\rightarrow\F_q$,
    we are done.

    $(2)\Rightarrow (3)$: Suppose $f$ is as in the hypothesis of the corollary,
    and let $U,\ell,b$ be as in this hypothesis.
    There exists a linearized polynomial $M$ of degree at most $p^{n-k-1}$ that induces the map $\ell$ on $U$ by Proposition \ref{prop:linmaplinpoly}.
    Let $U_i=U+v_i$ be the cosets of $U$.
    For any $i\in\intv{0}{p^k-1}$ and $\xi\in U_i$, write $\xi=u+v_i$ where $u\in U$, so that $M(\xi)=M(u)+M(v_i)=\ell(u)+M(v_i)=f(u)-b_i+M(v_i)=f(u)+a_i$
    where
    $a_i=M(v_i)-b_i$.

    $(3)\Rightarrow (1)$ is obvious.
\end{proof}

\begin{proposition}
\label{prop:ikik1}
    Let $\F_q$ be a field of characteristics $p$, where $q=p^n$ and $n\ge 1$. For any $0\leq k\leq n$,  let $I_k=I_k(q)$ be the number of self-mappings of $\F_q$ of codimension $k$.
    Then
    $$
    q^{p^k+n-k}\leq I_k(q)\leq q^{p^k+n-k+\min(n-k,k)}.
    $$
\end{proposition}
\begin{proof}
    We prove first the lower bound.
    Fix a subspace $U\leq \F_q$ of codimension $k$. The number of
    linear maps $\ell$ on $U$ is $q^{n-k}$. The number of tuples
    $a=(a_i)_{i=0,\ldots,p^{k}-1}$ is $q^{p^k}$.
    Two pairs $(\ell_1,a_1),(\ell_2,a_2)$ yieldings the same map
    must be equal because then $\ell_1-\ell_2=a_1-a_2$ is a constant linear map, whence both terms are 0.
    By Corollary \ref{cor:equivdef}, this implies the lower bound.

    Let us turn to the upper bound.
    A map of codimension $k$ is entirely determined by a triple $(U,\ell,a)$ where $U\leq \F_q$ is a subspace of codimension $k$,
    $\ell:U\rightarrow\F_q$ is linear, and $a\in\F_q^{p^k}$.
    We already counted pairs $(\ell,a)$ above. A space $U$ of codimension $k$ is entirely determined by a basis $(v_1,\ldots,v_{n-k})$, of which there are at most $q^{n-k}$. It is alternatively determined by a basis $(u_1,\ldots,u_k)$ of its orthogonal $U^\perp$ (with respect to a fixed inner product of $\F_q$). Thus, there
    are at most $q^{\min(n-k,k)}$ choices for $U$.
\end{proof}
\begin{corollary}
\label{cor:ikik1}
    As soon as $p>2$ or $k>2$, we have $I_k(q)\geq qI_{k-1}(q)$, in particular $I_{k-1}(q)=o(I_k(q))$ as $q$ tends to infinity.
    Further, in the regime where $p^k/n$ tends to infinity, we have $I_{k-1}(q)\ll_\varepsilon I_k(q)^{1/p+\varepsilon}$ for every $\varepsilon>0$.
\end{corollary}
\begin{proof}
    We use Proposition \ref{prop:ikik1}, which gives $I_k\geq q^{p^k+n-k}$
    and $I_{k-1}\leq q^{p^{k-1}+\min (2(n-k+1),k-1)}$.
    When $k-1\leq n/2$, we have $I_{k-1}\leq q^{p^{k-1}+n}$ and $p^k-p^{k-1}\geq k+1$ as soon as $p>2$ or $k>2$, whence $I_k(q)\geq qI_{k-1}(q)$ as desired. Otherwise $k-1<n/2$, so that 
    $n-k+1+\min (n-k+1,k-1)=2(n-k)<n-k+k$
    yet again $p^k-p^{k-1}\geq k+1$, which concludes the proof of the first point.

    For the second point, we write
    $
    I_{k-1}\leq I_k^{1/p}q^{2n}$
    and since $2n\leq \varepsilon p^k$ for all but finitely many terms, we conclude.
\end{proof}
\begin{corollary}
\label{cor:almostall}
    As $q$ tends to infinity, almost all self-mappings of $\F_q$ have least additive index $q$.
\end{corollary}
\begin{proof}
    Direct application of Corollary \ref{cor:ikik1},
    with $k=n$; note that either $p>2$ or $k=n>2$ since $q=p^n$ tends to infinity, so
    $I_{n-1}=o(I_n)$, where $I_n=q^q$ is the total number of self-mappings.
\end{proof}
\section*{Acknowledgments}

This research was funded by the Austrian Science Fund (FWF) [10.55776/PAT4719224].

\end{document}